\documentclass[a4paper]{article}

\usepackage{amsthm} \usepackage{amsfonts} \usepackage{amsmath}
\usepackage{diagrams} %\usepackage{showkeys}

\DeclareMathAlphabet{\mathbbb}{U}{bbold}{m}{n}
\DeclareMathAlphabet{\mathmanual}{U}{manfnt}{m}{n}
\newtheorem{thm}{Theorem}[section] \newtheorem{prop}[thm]{Proposition}
\newtheorem{lem}[thm]{Lemma} \newtheorem{cor}[thm]{Corollary}

\theoremstyle{definition} \newtheorem{dfn}[thm]{Definition}
\newtheorem{rem}[thm]{Remark} \newtheorem{ex}[thm]{Exercise}
\newtheorem{exam}[thm]{Example}

\newcommand{\calA}{\mathcal{A}} 
 
 \newcommand{\I}{\mathbbb{1}}
\newcommand{\Lef}{\mathbb{L}} \newcommand{\Z}{\mathbb{Z}} 
\newcommand{\calM}{\mathcal{M}} 
 
\newcommand{\calB}{\mathcal{B}} \newcommand{\bfD}{\mathbf{D}}

\newcommand{\barI}{\bar{\mathbbb{1}}} \newcommand{\bfC}{\mathbf{C}}
\newcommand{\qtc}{$\Q$-linear tensor category}
\newcommand{\qtcs}{$\Q$-linear tensor categories}
\newcommand{\qtf}{$\Q$-linear tensor functor}
\newcommand{\Sim}{\Sigma}
\newcommand{\reg}{\mathcal{O}}

\newcommand{\dleqn}{d_{\leq n}}
\newcommand{\calD}{\mathcal{D}}

\newcommand{\Spec}{\mathop{\mathrm{Spec}}}

\newcommand{\cone}{\mathop{\mathrm{cone}}}

\newcommand{\End}{\mathop{\mathrm{End}}}

\newcommand{\Ker}{\mathop{\mathrm{Ker}}}

\newcommand{\Img}{\mathop{\mathrm{Im}}}

\newcommand{\DM}{\mathbf{DM}} 

\newcommand{\A}{\mathbb{A}} \newcommand{\Proj}{\mathbb{P}}
\newcommand{\Q}{\mathbb{Q}}

 \newcommand{\D}{\mathbf{D}}

\newcommand{\tens}{\otimes} \newcommand{\isom}{\cong}

\DeclareFontFamily{U}{manual}{} 
\DeclareFontShape{U}{manual}{m}{n}{ <->  manfnt }{}

\begin{document}

\title{Schur functors and motives}

\author{Carlo Mazza\footnote{Partially supported
by INdAM ``Borse Di Studio Per L'Estero''}\\  
Universit\'e Paris 7 - Math\'ematiques, case 7012\\
\'Equipe Topologie et G\'eom\'etrie Alg\'ebriques\\
2 place Jussieu, 75251 Paris CEDEX 05\\
\texttt{mazza@math.jussieu.fr}}

\date{October 20th, 2004}

\maketitle

\begin{abstract}
In this article we study the class of Schur-finite motives, that is,
motives which are annihilated by a Schur functor. We compare this
notion to a similar one due to Kimura. In particular, we show that the
motive of any curve is Kimura-finite.  This last result has also been
obtained by V. Guletski\u{\i}. We conclude with an example by O'Sullivan
of a non Kimura-finite motive which is Schur-finite.
\end{abstract}

If $\lambda$ is a partition of $n$, the Schur functor $S_\lambda$
sends a motive $X$ to the direct summand of $X^{\tens n}$
determined by $\lambda$.  We say that $X$ is
Schur-finite if it is annihilated by some Schur functor.  This
definition is due to Deligne (\cite{delschur}) who related
Schur-finiteness to super-Tannankian categories. In this paper we
study Schur-finite objects in several categories, including motives.

Kimura and O'Sullivan have independently defined a stronger notion which we
will call Kimura-finiteness. Kimura showed in \cite{kimura} that if a
motive $M$ is Kimura-finite, then any endomorphism of $M$ is either
nilpotent or detected by cohomology. Kimura-finiteness
has been examined further by Guletski\u{\i} and Pedrini (see \cite{gp}
and \cite{gp2}) and by Andr\'e and Kahn (see \cite{yveskahn}).

In the first part of this paper we define Schur-finiteness and study
its properties in the setting of a \qtc. In particular, we investigate
its behavior with respect to tensor functors and triangles in the
derived category of an abelian category with tensor.  In the second
part we apply this formalism to the category of classical motives and
to Voevodsky's category $\mathbf{DM}^{eff,-}_{Nis}(k,\Q)$.  We
conclude by showing that the motives of all curves are Kimura-finite.
This last result has also been obtained by Guletski\u{\i} in
\cite{gul}. Using a result by O'Sullivan, we produce a motive which is
Schur-finite but not Kimura-finite.

The author would like to thank Chuck Weibel for everything he did. Much
credit is due to Pierre Deligne, who introduced the definition. We are
also grateful to Claudio Pedrini, Luca Barbieri Viale, and Bruno Kahn
for precious conversations and their comments. The author would like also to thank
the Instituto de Matem\'aticas of UNAM in Morelia, Mexico for hospitality while this
manuscript was prepared.

\section{Definitions and basic properties}

In this paper, $\calA$ will always be a \qtc, in the following sense. 

\begin{dfn}
% (See \cite[p.105]{dm}.) We say that a category $\calA$ is a 
% a symmetric monoidal category with identity if there is a bifunctor
% $\tens:\calA\times\calA \to \calA$ and an identity object which
% satisfy compatible associativity, commutativity and identity
% constraints.
We say that a symmetric monoidal category $\calA$ is a \textbf{{\qtc}} if it
satisfies all of the following:
\begin{enumerate}
\item $\calA$ is additive, pseudo-abelian, and $\Q$-linear,
\item $\tens$ is $\Q$-bilinear.
\end{enumerate}

Let $F:\calA\to \calB$ be a functor between two \qtcs. We say that $F$
is a \textbf{{\qtf}} if it is $\Q$-linear and it respects the symmetric
monoidal structures.
\end{dfn}

Recall that for every partition $\lambda$ of $n$ there is an idempotent
$c_\lambda\in \Q[\Sim_n]$ called the Young symmetrizer. If $\Sim_n$ acts
on an object $A$ of $\calA$, then there is an algebra map $\Q[\Sim_n]\to
\End( A)$. We will often confuse the element of $\Q[\Sim_n]$ with the
induced endomorphism of a representation.  Since
$c_\lambda^2=c_\lambda$ and $\calA$ is pseudo-abelian, $c_\lambda(A)$
is a direct summand of $A$.

\begin{dfn}
Let $\calA$ be a \qtc. The symmetric group $\Sim_n$ acts on $X^{\tens
  n}$ for every $X$. For every partition $\lambda$ of $n>0$, we define
$S_\lambda(X)=c_\lambda(X^{\tens n})$.  This assignment makes
$S_\lambda(-)$ into a functor, which we call the \textbf{Schur
  functor} of $\lambda$. In particular, we define
$Sym^n(X)=S_{(n)}(X)$ and $\wedge^n X=S_{(1,\ldots,1)}(X)$.
\end{dfn}

The following definitions are directly inspired by \cite{delschur}
and \cite{kimura}.

\begin{dfn}
An object $X$ of $\calA$ is called \textbf{Schur-finite} if there
is an integer $n$ and a partition $\lambda$ of $n$ such that $X$
is annihilated by the Schur functor of $\lambda$, i.e.,
$S_\lambda(X)=0$. It follows from \ref{sfgen} below
that $S_\mu(X)=0$ for all $\lambda\subseteq \mu$.

An object $X$ of $\calA$ is called \textbf{even} (respectively, \textbf{odd})
if there is an $n$ so that $\Lambda^n X=0$ (respectively, $Sym^n X=0$).
An object $X$ is called \textbf{Kimura-finite} if there is a
decomposition $X=X_+\oplus X_-$ such that $X_+$ is even and $X_-$ 
is odd. 

We will say that the category $\calA$ is Schur-finite (respectively, Kimura-finite)
if all objects of $\calA$ are Schur-finite (respectively, Kimura-finite).
% We will write $\Sf=\Sf(\calA)$ for the category of Schur-finite
% objects of $\calA$. Correspondingly, $\Kf=\Kf(\calA)$ will be the category of 
% Kimura-finite objects of $\calA$. 
\end{dfn}

Most of the basic results about Schur functors can be proved in the
setting of \qtcs. Let $H$ be a subgroup of $G$.
Then for any irreducible representation $W_i$ of $H$
and $V_j$ of $G$, we define $[res(V_j):W_i]$ to be the multiplicity
of $W_j$ inside the restriction $res(V_j)$ of $V_j$ to $H$.
In particular $res(V_j)=\oplus [res(V_j):W_i]W_i$. By Frobenius reciprocity,
the coefficient of $W_i$ is the same as $[Ind_H^G(W_i):V_j]$. 

Let $n_i$ be integers so that
$n_1+\ldots+n_r=n$ and consider $\Sim_{n_1}\times\ldots\times \Sim_{n_r}\subseteq
\Sim_n$. Let $\mu_i$ be a partition of $n_i$, and let $V_{\mu_i}$
be the corresponding irreducible representation. For every $\lambda$ partition of $n$,
and let $V_\lambda$ be the corresponding representation.
Then we define $[\lambda:\mu_1,\ldots,\mu_r]=[res (V_\lambda):
V_{\mu_1}\tens \ldots \tens V_{\mu_r}]=[Ind (V_{\mu_1}\tens \ldots \tens V_{\mu_r}):V_\lambda]$.

\begin{prop}\label{sfgen}
(See \cite[1.6-1.8]{delschur}.) Let $X$ and $Y$ be objects of $\calA$, then:
\begin{enumerate}
\item $S_\mu(X)\tens S_\nu(X)\isom \oplus [\lambda{:}\mu,\nu]\;S_\lambda(X)$,
where the sum is taken over all partitions $\lambda$ of $n=|\mu|+\nu|$;
\item if $S_\lambda(X)=0$, then $S_\mu(X)=0$ for all $\lambda\subseteq \mu$;
\item $S_\lambda(X\oplus Y)\isom \oplus 
[\lambda{:}\mu,\nu]\left(S_\mu(X)\tens S_\nu(Y)\right)$ where $|\mu|+|\nu|=|\lambda|$;
\item $S_\lambda(X\tens Y)\isom \oplus
[V_\mu\tens V_\nu{:}V_\lambda]\left(S_\mu(X)\tens S_\nu(Y)\right)$
where $|\mu|=|\nu|=|\lambda|$.
\end{enumerate}
\end{prop}

\begin{cor}\label{kfsf}
Kimura-finiteness and Schur-finiteness are closed under direct sums and
tensor products. Moreover, every Kimura-finite object is Schur-finite.
%that is, $\Kf\subseteq \Sf$.
\end{cor}

\begin{proof}
Kimura-finiteness and Schur-finiteness are preserved by $\oplus$ and $\tens$
by \ref{sfgen}. For Kimura-finiteness, this was already proven in \cite[5.11]{kimura}.

Now every Kimura-finite $X$ is a direct sum of two
Schur-finite objects. Since Schur-finiteness is closed under direct sums, $X$
is Schur-finite as well.
\end{proof}

Recall that if $\calA$ is a \qtc, then so is the category
$\calA^\pm$ of super-objects of $\calA$.

\begin{lem}\label{apm}
If $\calA$ is Schur-finite (respectively, Kimura-finite),
then the category $\calA^\pm$ is also Schur-finite (respectively, Kimura-finite).
\end{lem}

\begin{proof}
An object of $\calA^\pm$ is a pair $(V,W)$. Let us write $\barI$ for
$(0,\I)$. If $\lambda$ is a partition of $n$, then
$S_\lambda(V,0)=(S_\lambda V,0)$. However $S_\lambda(0,V)=
\bar\I^{n}\tens(S_{\lambda'}V,0)$ where $\lambda'$ is the transpose of
$\lambda$. But Schur-finiteness (respectively, Kimura-finiteness) is
closed under direct sums by \ref{kfsf}, and therefore we have the
statement.
\end{proof}

\begin{exam}\label{svs}
The category $Vect^\pm_k$ of finite-dimensional super-vector spaces
over a field $k$ of characteristic zero is a \qtc.
By \ref{apm}, every object is Kimura-finite. 
In fact, every vector space is (non-canonically)
isomorphic to $\I^p\oplus \bar\I^q$, 
and $S_\lambda(\I^p\oplus \bar\I^q)=0$ if and only if
the partition $\lambda$ contains the rectangle with $p+1$ rows and 
$q+1$ columns (see \cite[1.9]{delschur}).
\end{exam}

\begin{ex}\label{vb}
Let $X$ be a scheme and let $\mathbf{Vb}_X$ be the category of vector
bundles over $X$. Each vector bundle is even, so $(\mathbf{Vb}_X)^\pm$
is Kimura-finite by \ref{apm}.
\end{ex}

Before we proceed we need a technical lemma.

\begin{lem}\label{tech}
Let $\calA$ be an abelian $\Q$-linear category, and consider a short exact sequence
\[ 0\to A\to B\to C\to 0\]
of $\Sim_n$-equivariant maps. Then we have a short exact sequence
\[ 0\to c_\lambda(A)\to c_\lambda(B)\to c_\lambda(C)\to 0.\]
\end{lem}

\begin{proof}
Since the Young symmetrizer is an idempotent subfunctor
of the identity, a diagram chase yields the result. 
\end{proof}

\begin{lem}\label{ab}
(Cf. \cite[1.19]{delschur}.) Let $\calA$ be an abelian {\qtc} where the tensor
is right exact. Suppose we have a short exact sequence
$0\to A\to B\to C\to 0$, and that $B$ is Schur-finite (respectively
even, respectively odd). Then $C$ is Schur-finite (respectively even,
respectively odd). Moreover, if $A$ and $B$ are flat objects with respect 
to the tensor product, then $A$ is Schur-finite (respectively even, respectively odd).
\end{lem}

\begin{proof}
By \ref{tech}, each $c_\lambda(B^{\tens n})\to c_\lambda(C^{\tens n})$ is onto, and,
if $A$ and $B$ are flat, each $c_\lambda(A^{\tens n})\to
c_\lambda(B^{\tens n})$ is into.
\end{proof}

The presence of a {\qtf} between two {\qtcs} creates relations between
Schur-finite objects.

\begin{lem}\label{fssf}
Let $F:\calA\to \calB$ be a \qtf. 
If an object $X$ of $\calA$ is Schur-finite, so is $F(X)$. If
$F$ is also faithful, then the converse holds, i.e., if $F(X)$ is
Schur-finite, then so is $X$.
\end{lem}

\begin{proof}
The result follows from the fact that  
$F(S_\lambda(X))=S_\lambda(F(X))$ for all objects $X$ of $\calA$.
\end{proof}

\begin{exam}\label{dgalg}
Lemma \ref{fssf} fails for Kimura-finiteness. Let $\calA$ be the
category of graded modules over the graded algebra $A=k[\epsilon]$,
where $\epsilon^2=0$. The forgetful functor $F:\calA \to Vect^\pm_k$ sends
$A$ to $\I\oplus \bar \I$. Hence $A$ is neither even nor odd, but 
it is also indecomposable and therefore it is not Kimura-finite,
even though $F(A)$ is Kimura-finite.
However, $A$ is Schur-finite as
$S_{(2,2)}(A)=0$ (by \ref{svs} and \ref{fssf}).
\end{exam}

\begin{exam}\label{sfnotkfbis}
Let $\calA$ be the category of finitely generated $R$-modules, where
$R$ is a commutative $\Q$-algebra. Then $\calA$ is a \qtc, and by \ref{ab} all objects are
even. Now consider the category of bounded chain
complexes of finitely generated $R$-modules. We have a forgetful
functor from $\mathbf{Ch}^b(\calA)$ to the category $\calA^\pm$ of
super-objects of $\calA$. This forgetful functor is a faithful \qtf.
Since $\calA^\pm$ is Schur-finite by \ref{apm}, we
have that $\mathbf{Ch}^b(\calA)$ is Schur-finite by \ref{fssf}.
It need not be Kimura-finite; see \ref{sfnotkf}.
\end{exam}

\begin{rem}\label{homschur}
Let $\calM_h$ be the category of $\Q$-linear motives modulo
homological equivalence, for a fixed Weil cohomology $H$. By the
K\"unneth formula, the cohomology yields a faithful {\qtf}
$H:\calM_h\to Vect_\Q^\pm$. Since $Vect_\Q^\pm$ is
Schur-finite by \ref{svs}, $\calM_h$ is Schur-finite by \ref{fssf}.
Let $\calM^\pm_h$ be the subcategory of motives of $\calM_h$
for which the odd and the even part of the
%Under the additional assumption that the components of the K\"unneth
decomposition of the diagonal are algebraic.
Y.  Andr\'e and B. Kahn in \cite[9.2.1c+B.2]{yveskahn}, and
independently P. O'Sullivan, proved that $\calM^\pm_h$ is Kimura-finite.

Let $\calM_n$ be the category of $\Q$-linear motives modulo numerical
equivalence.  Since we have a {\qtf} from $\calM_h$ to $\calM_n$ and
$\calM_h$ is Schur-finite, $\calM_n$ is Schur-finite by \ref{fssf}.

Kimura (and O'Sullivan independently) conjectured in \cite[7.1]{kimura} that the category $\calM_r$
of $\Q$-linear motives modulo rational equivalence is Kimura-finite.
This conjecture combined with \ref{fssf} implies that $\calM_h$ is
Kimura-finite.
\end{rem}

\begin{thm}
The category $\calM_n$ is super-Tannakian, i.e.,
there exists a field $K$ of characteristic zero and a faithful fibre functor
from $\calM_n$ to $Vect^\pm_K$.
\end{thm}

\begin{proof}
The category $\calM_n$ is abelian and semi-simple.% by \cite{jannsen}.
This in particular implies that the tensor product is exact.  This
category is also rigid by \cite[p. 232]{jannsenbanff}. Therefore the
result comes from \ref{homschur}, \cite[2.1]{delschur}, and 
the fact that every commutative ring maps to a field. The
faithfulness comes automatically from the rigidity (see
\cite[0.9]{delschur}).
\end{proof}

\section{Abelian \qtcs}

In this section we assume that the {\qtc} $\calA$ is abelian and study
how Schur-finiteness behaves with respect to extensions.

The following construction is adapted from \cite[1.19]{delschur} and
will be useful to prove that Schur-finiteness is closed under
extension of flat objects. 

Let $\calA$ be an abelian tensor category and let
$X$ be the extension
\[ 0\rTo P\rTo X \rTo Q\rTo 0\]
where $P$ and $Q$ are flat objects. Then we define a
$\Sim_n$-equivariant filtration of $X^{\tens n}$ as follows.  The
filtration $F_i(X^{\tens n})$ will be the subobject generated by all
$n$-fold tensor products where $n-i$ factors are copies of $P$ and the
remaining $i$ are copies of $X$. To make this precise, we establish
some notations.

\begin{dfn}\label{dfnfiltration}
Let $X$ be the extension
\[ 0\rTo P\rTo X \rTo Q\rTo 0.\]
For every pair of numbers $i$ and $j$ so that $i+j=n$,
we define
\[T_{j,i}(P,X)=Ind_{\Sim_j\times \Sim_i}^{\Sim_n}(P^{\tens j}\tens X^{\tens i}).\]
The $\Sim_j\times \Sim_i$-equivariant maps $P^{\tens j}\tens X^{\tens
  i}\to X^{\tens n}$ induce $\Sim_n$-equivariant maps
$f_i:T_{j,i}(P,X)\to X^{\tens n}$. We define
\[F_i(X^{\tens n})=\Img(f_i)= T_{j,i}(P,X)/(\Ker f_i).\]
In particular $F_0(X^{\tens n})=f_0(T_{n,0}(P,X))=P^{\tens n}$,
$F_n(X^{\tens n})=f_n(T_{0,n}(P,X))=X^{\tens n}$ and 
\[ T_{n-1,1}(P,X)=
(P^{\tens n-1}\tens X) \oplus  (P^{\tens n-2}\tens X\tens P)\oplus \ldots
\oplus (X\tens P^{\tens n-1}).\]
% We define the maps among these objects in such a way that the following is a resolution
% of $Q$:
% \[ 0\to T_{n,0}(P,X)\to T_{n-1,1}(P,X)\to \ldots \to T_{1,n-1}(P,X)\to T_{0,n}(P,X)\to 0.\]
\end{dfn}

Since the maps $f_i$ are
$\Sim_n$-equivariant, so are the $F_i(X^{\tens n})$.  Since the map $P^{j+1}\tens
X^{i-1}\to X^{\tens n}$ factors through $P^j\tens X^i$, then 
the $\Sim_n$-equivariant map $f_{i-1}$
factors through $f_i$, and hence $F_{i-1}(X^{\tens n})=\Img(f_{i-1})\subseteq
\Img(f_i)=F_i(X^{\tens n})$.  Therefore the $F_i(X^{\tens n})$ form a $\Sim_n$-equivariant
filtration of $X^{\tens n}$.

% \begin{lem}\label{kerfi}
% The kernel of the map $f_i$ is generated by $n$-fold tensor products, where
% $j+1$ terms are copies of $P$ and $i-1$ terms are copies of $X$.
% \end{lem}

% \begin{proof}
% Let us first address the first case when $n=2$ and $i=j=1$. Since $P$ and $Q$
% are flat, we have an exact sequence
% \[ 0\rTo P^2\rTo^{(+,-)} (P\tens X)\oplus (X\tens P)\rTo X^2\rTo Q^2.\]

% In the general case, $T_{j,i}(P,X)$ is a direct sum of
% $l=\binom{n}{i}$ $n$-fold tensor products. The map $f_i$ factors as
% \[ T_{j,i}(P,X)\rTo^r \bigoplus^{l} X^n\rTo^{g_i} X^n.\]
% The map $r$ is injective and therefore $\Ker(f_i)=r^{-1}(\Ker(g_i))$.

% The kernel of $g_i$ is generated by the $l-1$ copies of $X^n$
% embedded in $\oplus^l X^n$ by the maps
% $x\mapsto (0,\ldots,x,0\ldots,0,-x)$.  But now we essentially
% reduced to the previous case, and therefore we conclude.\marginpar{fix this}

% %% We claim that it is generated
% %% by a direct sum of objects chosen as follows. Consider a pair of direct
% %% summands $S_1$ and $S_2$ of $T_{j,i}(P,X)$ which differ only for the
% %% exchange of a $P$ and an $X$, i.e., if $S_1=\ldots \tens P \tens
% %% \ldots \tens X \tens \ldots$ then $S_2=\ldots \tens X\tens \ldots
% %% \tens P\tens \ldots$.  Then pick as a generator the object $\ldots
% %% \tens P\tens \ldots \tens P\tens \ldots$ which is sent with a plus
% %% sign to $S_1$ and with a minus sign to $S_2$.
% \end{proof}

\begin{prop}\label{filtration}
Let $X$ be an extension of two flat objects $P$ and $Q$ and let
$F_i(X^{\tens n})=\Img(f_i)$ as in \ref{dfnfiltration}. Then $F_i/F_{i-1}\isom
T_{n-i,i}(P,Q)$.
\end{prop}

\begin{proof}
We are going to proceed by induction on $n=i+j$. For $n=1$ it is clear.
Let us suppose that the statement is true for $n-1$ and
consider the filtration $F_*(X^{n-1})$ on $X^{n-1}$
and the given filtration $F_*(X)$ on $X$. The tensor product of the
two filtrations yields the filtration $F_*(X^{\tens n})$ on $X^{\tens n}$.
Set $gr_i(X^{\tens n})=F_i(X^{\tens n})/F_{i-1}(X^{\tens n})$. Since $X$ is flat,
$gr_*(X^{i-1})\tens gr_*(X)=gr_*(F_*(X^{n-1}))\tens gr_*(F_*(X))\isom gr_*(F_*(X^{\tens n}))
=gr_*(X^{\tens n})$ by \cite[Ex. III.2.6]{bourbaki}.
But $gr_i(X)$ is $P$ when $i=0$, $Q$ when $i=1$, and $0$ otherwise. Therefore
\[ gr_i(X^{\tens n})\isom \left(gr_i(X^{n-1})\tens P\right) \oplus 
\left(gr_{i-1}(X^{n-1})\tens Q\right)\]
\[= \left(T_{n-1-i,i}(P,Q)\tens P\right) \oplus 
\left(T_{n-i,i-1}(P,Q)\tens Q\right)=T_{i,j}(P,Q).\qedhere\]
\end{proof}

And now to the theorem we advertised before.

\begin{thm}\label{abext}
(Cf. \cite[1.19]{delschur}.) Let $\calA$ be an abelian \qtc. Then any 
extension of Schur-finite flat objects is Schur-finite.
\end{thm}

\begin{proof}

Consider the extension
\[ 0\to P\to X\to Q\to 0\]
and the corresponding filtration constructed in \ref{dfnfiltration}.
Choose an integer $n$ and a partition $\lambda$ of $n$ so that
$S_\lambda(P)=S_\lambda(Q)=S_\lambda(P\oplus Q)=0$.  By
\ref{filtration}, $\oplus_i F_i/F_{i-1}=(P\oplus Q)^{\tens n}$.  Since
$F_i/F_{i-1}$ is $\Sim_n$-invariant, $\oplus
c_\lambda(F_i/F_{i-1})=c_\lambda (\oplus F_i/F_{i-1})= c_\lambda (
(P\oplus Q)^{\tens n})=S_\lambda(P\oplus Q)=0$ and therefore
$c_\lambda(F_i/F_{i-1})=0$ for every $i$.
Consider the short exact sequences
\[ 0\to F_{i-1}\to F_i\to F_i/F_{i-1}\to 0.\]
By \ref{tech}, we have short exact sequences
\[ 0\to c_\lambda(F_{i-1})\to c_\lambda(F_i)\to c_\lambda(F_i/F_{i-1})\to 0.\]
But $c_\lambda(F_0)=c_\lambda(P^{\tens n})=S_\lambda(P)=0$ by hypothesis.
Hence we proceed by induction to prove that $c_\lambda(F_i)=0$ for
every $i$. In particular, $S_\lambda(X)=c_\lambda(F_n)=0$.
\end{proof}

\begin{cor}\label{extevenodd}
Let $\calA$ be an abelian \qtc. 
Let $0\to P\to X\to Q\to 0$
be a short exact sequence of flat objects.
If $S_\lambda(P\oplus Q)=0$, then $S_\lambda(X)=0$.
In particular, an extension of odd (respectively, even) 
flat objects is odd (respectively, even).
\end{cor}

\begin{proof}
This is clear from the proof of \ref{abext}.%, since if $P$ and $Q$ are
%even (respectively, odd) then $P\oplus Q$ is even (respectively, odd).
\end{proof}

\section{Chain complexes and derived categories}

In this section we will always assume that $\calA$ is an abelian \qtc.
Consider the category $\bfC=\mathbf{Ch}^-(\calA)$ of bounded below chain
complexes. If $M$ and $N$ are two objects, then we define
$M\tens_\bfC N=Tot^\oplus(M\tens N)$. Using this tensor product,
$\bfC$ is an abelian {\qtc}.

It is easy to see that all Schur-finite complexes in $\mathbf{Ch}^-(\calA)$ are
bounded. Therefore we will be interested in bounded complexes.

\begin{lem}\label{chsf}
If $\calA$ is abelian and Schur-finite, then $\mathbf{Ch}^b(\calA)$ is Schur-finite.
\end{lem}

\begin{proof}
Consider the faithful forgetful {\qtf}
$\mathbf{Ch}^b(\calA)\to \calA^\pm$. Since $\calA$ is Schur-finite,
then so is $\calA^\pm$ by \ref{apm}. Then \ref{fssf} yields the result.
\end{proof}

The category of chain complexes provides us with another example of
an object which is Schur-finite but not Kimura-finite, beside \ref{dgalg}.

\begin{exam}\label{sfnotkf}
(B. Kahn) Consider the category of bounded chain complexes of $R$-modules,
where $R=\Q[x]$. This is clearly a \qtc. Let $M$ be the complex
$R\rTo^x R$. This complex is irreducible, and is not Kimura-finite because 
$Sym^n M\isom M$ and $\wedge^n M\isom M[n-1]$.   
By \ref{chsf}, $M$ is Schur-finite.
\end{exam}

\begin{exam}
(P. O'Sullivan) Consider the category of chain complexes of coherent
modules for $\Proj^1$. By \ref{vb} and \ref{sfnotkf}, the complex
$\reg(1)\rTo^x \reg$ is Schur-finite but not Kimura-finite. This category
has the feature that $\End(\reg)=k$.
\end{exam}

%Consider now the categories $\bfK=\mathbf{K}^-(\calA)$ of bounded below chain
%complexes where we mod out by chain homotopy equivalence, 
Let $\bfD=\mathbf{D}^-(\calA)$ be the bounded below derived category
and consider the localization functor $q:\bfC\to \bfD$.
%\[ \bfC\rTo^k \bfK\qquad \bfK\rTo^d \bfD\qquad \bfC\rTo^q \bfD.\] 
For simplicity, let us assume that $\calA$ has enough projectives
to avoid some technical difficulties. 
%Both $\bfK$ and $\bfD$ have a tensor product.  The tensor product of $\bfK$ is defined exactly like
%the one in $\bfC$, but 
In the derived category we define the
tensor product of two objects $M$ and $N$ as
\[ M\tens_\bfD N=Tot^\oplus(P\tens Q),\]
where $P$ and $Q$ are two projective resolutions of $M$ and $N$,
respectively.  (See \cite[Lec. 8]{notes}.)

With these conventions, if $\calA$ is an abelian \qtc, then the
bounded below derived category $\bfD$ is a \qtc.

\begin{lem}\label{shurproj}
If $P$ is a complex of projectives, then
$q(S_\lambda(P))=S_\lambda(q(P))$.  In particular, if
$S_\lambda(P)=0$ in $\bfC$, then $S_\lambda(P)=0$ in $\bfD$.
Conversely, if $S_\lambda(P)=0$ in $\bfD$, then $S_\lambda(P)$ is
acyclic in $\bfC$.
\end{lem}

\begin{proof}
Clear from the fact that if $P$ is a complex of projectives, then
$q(P\tens_\bfC P)=P\tens_\bfD P$.
\end{proof}

\begin{lem}
Let $\calA$ be Schur-finite and
let $X$ be a bounded complex with a finite projective resolution. Then
$X$ is Schur-finite in $\mathbf{D}^b(\calA)$.
\end{lem}

\begin{proof}
Let $P$ be a finite projective resolution. By \ref{chsf}, there is a
$\lambda$ so that $S_\lambda(P)=0$ in $\bfC$. By \ref{shurproj},
$S_\lambda(P)$ is zero in $\bfD$ as well.
\end{proof}

\begin{lem}\label{2of3}
%Let $\calD$ be a triangulated {\qtc}. Then Schur-finiteness has 
%the two out of three property.
Let $\bfD$ be the derived category of an abelian {\qtc} $\calA$. Then
Schur-finiteness has the two out of three property.
\end{lem}

\begin{proof}
Let us consider the triangle $A\to B\to C\to A[1]$. Without loss of
generality, we may assume that $A$ and $B$ are Schur-finite, and we
need to prove that $C$ is such. We may replace $A$ by a projective
resolution $P$, and similarly we replace $B$ with $Q$. 
If $f:P\to Q$, we may assume that $C$ is just $\cone(f)$
and we have a short exact sequence
\begin{equation}\label{conefiltr}
0\to Q\to \cone(f)\to P[1]\to 0
\end{equation}
We will show that $\cone(f)$ is Schur-finite in $\bfD$, i.e.,
$S_\lambda(\cone(f))$ is acyclic for some $\lambda$.

Choose a partition $\lambda$ of $n$ such that the complexes
$S_\lambda(P[1])$, $S_\lambda(Q)$, and $S_\lambda(Q\oplus P[1])$ are
all acyclic.  Since $\calA$ is an abelian \qtc, so is $\bfC$.
Therefore we may use \ref{dfnfiltration} to define a $\Sim_n$-equivariant
filtration $F_i$ of $(\cone(f))^{\tens n}$ coming from the short exact sequence
(\ref{conefiltr}). By \ref{tech},
we have short exact sequences
\[ 0\to c_\lambda(F_{i-1})\to c_\lambda(F_{i})\to c_\lambda(F_i/F_{i-1})\to 0.\] 
Since the sequence (\ref{conefiltr}) splits in every degree,
it still splits degreewise when we tensor with any other object.
Therefore the proof of \ref{filtration} goes through to give that 
$\oplus F_i/F_{i-1}= (Q\oplus P[1])^{\tens n}$. Since
each $F_i/F_{i-1}$ is $\Sim_n$-invariant, we have that $ \oplus
c_\lambda(F_i/F_{i-1}) =c_\lambda (\oplus (F_i/F_{i-1}))= c_\lambda
((Q\oplus P[1])^{\tens n})=S_\lambda(Q\oplus P[1])$, which is acyclic. 
This forces all $c_\lambda(F_i/F_{i-1})$ to be acyclic.  By
hypothesis $c_\lambda(F_0)=S_\lambda(P)$ is acyclic, and therefore 
it follows by recursion that each $c_\lambda(F_i)$ is acyclic. In
particular $c_\lambda(F_n)=S_\lambda(\cone(f))$ is acyclic, so
$S_\lambda(\cone(f))=0$ in $\bfD$.
\end{proof}

\begin{cor}\label{2of3evenodd}
Let $\bfD$ be the derived category of an abelian {\qtc} $\calA$
and let $A\to B\to C\to A[1]$ be a triangle.
If $S_\lambda(A\oplus C)=0$, then $S_\lambda(B)=0$.
In particular, if $A$ and $C$ are odd (respectively, even), then $B$ is odd 
(respectively, even.)
\end{cor}

\begin{proof}
Clear from the proof of \ref{2of3}.
\end{proof}

\begin{rem}
Lemma \ref{2of3} does not hold for Kimura-finiteness. Let $\bfD$
be the derived category category of finitely generated $\Q[x]$-modules.
We have a triangle $\Q[x]\rTo^x \Q[x]\rTo \Q\rTo \Q[x][1]$
but $\Q$ is not Kimura-finite by \ref{sfnotkf}.
\end{rem}

\begin{exam}
Consider the category $Sh(X)$ of coherent sheaves on a quasi-projective
scheme $X$ over a field containing $\Q$. Replacing ``projective'' 
by ``locally free'', the proofs of \ref{shurproj}-\ref{2of3evenodd} go through.
\end{exam}

\begin{exam}\label{shnis}
Let $\calA=Sh_{Nis}(Cor_k,\Q)$ denote the category of Nisnevich
sheaves of $\Q$-modules with transfers on $Sm/k$. By \cite{notes}
or \cite[p. 206]{tri}, this is an abelian \qtc, and so is $\D^-(\calA)$.
Replacing ``projective'' by ``representable'', we again see that the proofs of 
\ref{shurproj}-\ref{2of3evenodd} go through.
\end{exam}

\section{Applications to Chow motives}

%Let us recall the basic facts about (classical) Chow motives.

%\begin{dfn}
% Let $\calM_r$ denote the category of classical Chow motives (see
% \cite{scholl}). That is, $\calM_r$ is the pseudo-abelian completion of
% the category whose objects are smooth projective varieties over $k$
% and morphisms are rational equivalence classes of $\Q$-linear cycles
% of the appropriate codimension.

For any adequate equivalence relation (see \cite{jannsenbanff}),
we can construct a category of $\Q$-linear motives. They are 
all {\qtcs} and therefore the notions of Schur-finiteness and
Kimura-finiteness make sense (see \ref{homschur} above). 
Kimura-finiteness has been studied in \cite{kimura}, \cite{yveskahn},
\cite{gp}, \cite{gp2}, and \cite{Gul2}.

% Similarly, if we use a different adequate (see \cite{jannsenbanff}) equivalence relation, we
% would get different categories of motives: $\calM_a$ using algebraic
% equivalence, $\calM_h$ using homological equivalence (for a fixed Weil
% cohomology), and $\calM_n$ using numerical equivalence. They are all
% \qtcs. 
%\end{dfn}

% \begin{rem}
% By \cite{jannsen} $\calM_n$ is an abelian semi-simple category and it is the
% only such category, i.e., the other equivalence relations do not give
% an abelian semi-simple category.

% Notice also that by \cite[Cor 3.5]{scholl}, if $k$ is not in the algebraic
% closure of a finite field, the category $\calM_{r}$ is not abelian.
% \end{rem}

% \begin{rem}
% % By \cite[p. 232]{jannsenbanff}, the category $\calM_\sim$ is
% % rigid for every adequate equivalence relation.
% If $M$ is the motive of a smooth projective variety $X$,
% the Young symmetrizer is the projector
% \[ d_\lambda=\frac{\dim V_\lambda}{n!}\sum_{g\in\Sigma_n}\chi_\lambda(g)
% \Gamma_g.\]
% So $S_\lambda(X)=(X,d_\lambda)$. Moreover, if $M=(X,p)$,
% then $S_\lambda(M)=(X,p\circ d_\lambda)$.
% \end{rem}

% \begin{dfn}
% Consider and object $X$ in $\calT$ and an endomorphism $f$ of $X$.
% We define the trace of $f$ to be the composition
% \[ \I\rTo X^\vee\tens X\rTo^{id\tens f}X^\vee\tens X=X\tens X^\vee \rTo \I.\]
% We define the rank, or the dimension, of $X$ to be $rk(X)=Tr(id_X)$.
% \end{dfn}

\begin{exam}\label{motpn}
It is well known that $M(\Proj^1)=\I\oplus \Lef$. Clearly, the switch
acts as the identity on both $\I\tens \I$ and
$\Lef\tens \Lef$,  and therefore $\Lambda^2\I=\Lambda^2\Lef=0$.
Thus, 
\[ S_\lambda(\Proj^1)=0 \text{ iff $\lambda\supseteq (1,1,1)$}.\]
A similar argument shows that
the motive of $\Proj^n$ is Kimura-finite.
\end{exam}

We cite the following result without proof.

\begin{prop}\label{spcurvekf}
(See \cite[Corollary 4.4]{kimura}.) The motive of every smooth
projective curve is Kimura-finite.
\end{prop}

\begin{lem}
Let $M$ and $N$ be two motives and let $\lambda$ be a partition of
$d$. If $\Lambda^2(N)=0$, then $S_\lambda(M\tens N)=S_\lambda(M)\tens
N^{\tens d}$.
\end{lem}

\begin{proof}
By \ref{sfgen},
\[ S_\nu(M\tens N)=\oplus C_{\lambda,\eta,\nu} S_\lambda(M)\tens S_{\eta}(N),\]
where $|\lambda|=|\eta|=|\nu|=d$ and
$C_{\lambda,\eta,\nu}=[V_\lambda\tens V_\eta:V_\nu]$. In this
case, $S_{\eta}(N)=0$ for all partitions $\eta$ of $d$ except
for $S_{(d)}(N)=N^{\tens d}$.  But $C_{\lambda,(d),\nu}=
C_{\lambda,\nu,(d)}$ which is 1 if $\lambda=\nu$ and it is zero
otherwise. The result is proved.
\end{proof}

\begin{cor}\label{mtensl}
A motive $M$ is Schur-finite if and only if $M\tens \Lef$ is Schur-finite.
\end{cor}

\begin{cor}\label{bup}
Let $X_Y$ be the blowup of a smooth projective variety $X$ along a
pure codimension $r$ smooth subvariety $Y$. If $M(X_Y)$ is
Schur-finite, then both $M(X)$ and $M(Y)$ are Schur-finite.
Conversely, if $M(X)$ and $M(Y)$ are both Schur-finite, then $M(X_Y)$ is
Schur-finite.
\end{cor}

\begin{proof}
Just recall from \cite{manin} that if $X_Y$ is the blowup of a variety $X$
along a pure codimension $r$ subvariety $Y$, we have that:
\begin{equation*}
M(X_Y)=M(X)\oplus \left( \bigoplus_{i=1}^{r-1} M(Y)\tens \Lef^{\tens i}\right).\qedhere
\end{equation*}
\end{proof}

\begin{cor}\label{birat}
Schur-finiteness is a birational invariant for smooth projective surfaces.
% The motive $M$ of a surface $X$ is Schur-finite if and only if 
% the motive $M'$ of a blowup $X'$ of $X$ is Schur-finite.
\end{cor}

\begin{proof}
If two surfaces are birationally equivalent, then there
is a sequence of blow-ups and blow-downs along points which connects them.
% Let $M$ be the motive of a surface $X$ and let $M'$ be the motive of a
% surface $X'$ birationally equivalent to $X$.  It will suffice to check
% the case when $X'$ is the blowup of $X$ at a point. In this case, the
% usual (ref!) formula for monoidal transformations, we have
% that\marginpar{B-K paper 4.8}
% \[ M\isom M'\oplus \Lef\]
% which clearly implies the statement.
\end{proof}

% \begin{rem}
% Consider a Weil cohomology functor $H:\calM_r\to Vect_\Q$.  Since both
% categories are rigid symmetric monoidal categories, if $F$ is a
% symmetric monoidal functor, then $F(rk(X))=rk(F(X))$ for all objects
% $X$. However $H(rk(X))=\sum_i (-1)^i\dim H^i(X)$ (see the proof of
% Corollary 2 in \cite{jannsen})  and therefore, $H$ is not a symmetric
% monoidal functor.
% \end{rem}

The Kimura-finite analogues of \ref{mtensl}, \ref{bup} and \ref{birat}
were established in \cite{kimura}.

\section{Applications to the category $\DM$}

Let $\DM=\mathbf{DM}^{eff,-}_{Nis}(k,\Q)$ be the tensor triangulated
category of $\Q$-linear motives, i.e., the localization by $\A^1$-weak
equivalences of the derived category
of (cochain) complexes of Nisnevich sheaves
$\D^-=\mathbf{D}^-(Sh_{Nis}(Cor_k,\Q))$.  Recall that the tensor
structure is given by the localization of $\tens^{tr}_{L,Nis}$
(see \cite[9.5, 14.2 and 14.22]{notes}). We
write $q$ for the localization functor $\D^-\to \DM$
and $i$ for the adjoint embedding $\DM\to \D^-$. By \ref{shnis} $\D^-$ is
a \qtc, and therefore $\DM$ is a {\qtc} as well.

\begin{rem}
By \cite[p. 197]{tri} we have a faithful {\qtf} from the category
of classical Chow motives $\calM_r$
to $\DM$. Therefore, proving that a motive is Schur-finite in 
$\DM$ is equivalent to proving it in the category $\calM_r$
of Chow motives.
\end{rem}

\begin{lem}\label{basechange}
Let $M$ be a motive over $k$ and let $K$ be a finite 
extension of $k$. Let $M_K$ be the corresponding motive
over $K$. If $M_K$ is Schur-finite (respectively, Kimura-finite),
then $M$ is Schur-finite (respectively, Kimura-finite).
\end{lem}

\begin{proof}
Since we are working with $\Q$ coefficients, the proof of \cite[1.12]{notes} 
goes through in this setting and we have $\Q$-linear adjoint functors
\[ PreSh(Cor_k,\Q)\rTo^\phi PreSh(Cor_K,\Q)\rTo^\psi 
PreSh(Cor_k,\Q),\]
where $M_K=\phi(M)$ and $M$ is a direct summand of $\psi(M_K)$. 
Since $\psi$ is a \qtf, the result for Schur-finiteness 
follows from Lemma \ref{fssf}. In the Kimura-finite case,
we conclude by \cite[3.11]{gp}.
\end{proof}

Note that $q$ is a \qtf, but it is not faithful. Therefore if
$S_\lambda(qA)=0$ in $\DM$, then $S_\lambda(A)$ need only be
$\A^1$-weak equivalent to $0$ in $\D^-$. Note also that $i$ is not a \qtf.

\begin{prop}\label{2of3dm}
Schur-finiteness has the two out of three property in $\DM$.
\end{prop}

\begin{proof}
Consider the triangle $A\to B\to C\to A[1]$ in $\DM$. We may assume
that $A$ and $B$ are Schur-finite, and we need to prove that $C$ is
such.  Choose an integer $n$ and
a partition $\lambda$ of $n$ such that
$S_\lambda(A[1]\oplus B)=0$ in $\DM$. We will show that
$S_\lambda(C)=0$ in $\DM$.

Applying $i$ to the triangle above yields a triangle in $\D^-$,
but $S_\lambda(iA[1]\oplus 
iB)$ may only be $\A^1$-weakly equivalent to $0$ in $\D^-$. Let us replace $A$
and $B$ by quasi-isomorphic complexes $P$ and $Q$, respectively, which
are sums of representables of the form $\mathbb{Q}_{tr}(X)$ in each degree. If $f:P\to Q$, we
may assume that $C$ is just $\cone(f)$ and we have a short exact
sequence in $\bfC=\mathbf{Ch}^-(Sh_{Nis}(Cor_k,\Q))$
\[ 0\to Q\to \cone(f)\to P[1]\to 0.\]
By \ref{shnis}, $\bfC$ is an abelian {\qtc}. 
Consider the filtration $F_*$ of $\cone(f)^{\tens n}$ given by
\ref{dfnfiltration}. By \ref{tech}, we have short exact sequences in
$\bfC$
\[ 0\to c_\lambda(F_{i-1})\to c_\lambda(F_i)\to c_\lambda(F_i/F_{i-1})\to 0.\]
We know by hypothesis that $c_\lambda(F_0)=c_\lambda((P[1])^{\tens
  n})$ is $\A^1$-weakly equivalent to zero. By \ref{filtration}, all
$c_\lambda(F_i/F_{i-1})$ are $\A^1$-weak equivalent to zero.  Hence we
use induction to conclude that $S_\lambda(\cone(f))= c_\lambda(F_n)$
is $\A^1$-weakly equivalent to zero. But then $S_\lambda(\cone(f))=0$ in $\DM$.
% But $qc_\lambda(A^{\tens_{\D^-} n})=c_\lambda(A^{\tens_\DM n})$
% and so we can again use recursion to conclude.
%% We know we can identify the category $\DM$ with the subcategory
%% $\mathcal{H}$ of complexes with homotopy invariant cohomology of
%% $\D^-$. Therefore, we can think of each triangle in $\DM$ as a
%% triangle in $\D^-$. But $\cal{H}$ is a triangulated subcategory, so
%% every triangle in $\DM$ is isomorphic to a triangle in $\calH$ where
%% the proof of \ref{2of3} still holds.
\end{proof}

\begin{cor}\label{sfthick}
The subcategory of $\DM$ consisting of Schur-finite objects
is thick and closed under twists.
\end{cor}

\begin{cor}\label{dmevenodd}
Let $A\to B\to C\to A[1]$ be a triangle in $\DM$.
If $S_\lambda(A\oplus C)=0$, then $S_\lambda(B)=0$.
In particular, if $A$ and $C$ are even (respectively, odd) then
$B$ is even (respectively, odd).
\end{cor}

\begin{proof}
Clear from the proof of \ref{2of3dm}.
\end{proof}

With these results available, we can prove that the motive
of every curve is Kimura-finite.

From now on we will write $\Lef$ for $\mathbb{Q}(1)[2]$
to lighten the notations.

\begin{lem}\label{splitl}
Let $P$ be a smooth rational point on a projective curve $X$.
Then the following is a split triangle:
\[ M(X-P)\rTo M(X)\rTo \Lef\rTo^0 M(X-P)[1].\]
\end{lem}

\begin{proof}
This is obtained from the triangles on p. 196 in \cite{tri}.
\end{proof}

\begin{prop}\label{smoothcurvekf}
The motive of a smooth curve is Kimura-finite.
\end{prop}

\begin{proof}
Let $X$ be a smooth curve. There exists a smooth projective curve $\bar X$
and an open embedding $X\rInto \bar X$, such that the complement is a collection
of smooth points $P_0,\ldots,P_n$. By base change \ref{basechange} we may assume that
all the points $P_i$ are rational and that $X$ contains a smooth rational point $Q$.

First consider $X'=\bar X-P_0$.
By \ref{splitl}, we have a triangle
\[ \Lef[-1]\rTo^0 M(X')\rTo M(\bar X)\rTo \Lef\]
from which we can split off the motive of the rational point $Q$ and,
writing $\widetilde M(X)$ for the reduced motive $M(X)/M(Q)$ of $X$, get
\[ \Lef[-1]\rTo^0 \widetilde M(X')\rTo \widetilde M(\bar X) \rTo \Lef.\]
Since $\bar X$ is a smooth projective curve, the reduced motive
decomposes as $\widetilde M(\bar X)=M_1(X)\oplus \Lef$, where $M_1(\bar
X)$ is odd.  By \ref{splitl}, we may split off the copy of $\Lef$, and
get that $M_1(\bar X)\isom \widetilde M(X')$. Since $M_1(\bar X)$ is
odd, so is $\widetilde M(X')$.

By \cite[p. 196]{tri}, we have a triangle
\[ \oplus_1^{n}\Lef[-1]\rTo M(X)\rTo M(X')\rTo \oplus_1^n \Lef\]
Splitting off the motive of the point $Q$, we get 
\[ \oplus_1^n\Lef[-1]\rTo \widetilde M(X)\rTo \widetilde M(X') \rTo \oplus_1^n\Lef.\]
But $\widetilde M(X')$ is odd by the first part of this proof 
and clearly $\oplus_1^n \Lef[-1]$ is odd, so by \ref{dmevenodd} $\widetilde M(X)$
is odd too. But $M(X)=\I\oplus \widetilde M(X)$, and therefore we have the statement.
\end{proof}

\begin{thm}\label{curvekf}
The motive of any curve is Kimura-finite.
\end{thm}

\begin{proof}
The smooth case was established in \ref{smoothcurvekf}. Suppose
that $X$ is a singular affine curve. Let $Z$ be the singular locus
of $X$, and let $X'$ be the normalization. Then we have the cartesian diagram
\begin{diagram}
Z'&\rTo&X'\\
\dTo&&\dTo\\
Z&\rTo& X.
\end{diagram}
By \cite[Prop. 4.1.3]{tri}, we have a triangle
\[ M(Z')\to M(Z)\oplus M(X')\to M(X)\to M(Z')[1].\]
By base change \ref{basechange}, we may assume that both $Z$ and $Z'$ consist of rational 
points. 
Let $K$ be the kernel of the map $M(Z')\to M(Z)$ and note that
$M(Z')\isom K\oplus M(Z)$. Then the triangle becomes
\[ K\to M(X')\to M(X)\to K[1].\]
By base change \ref{basechange}, we may assume that $X$ contains a smooth rational point
which we can split off, and get a triangle
\[ K\to \widetilde M(X')\to \widetilde M(X)\to K[1].\]
By the proof of \ref{smoothcurvekf}, $\widetilde M(X')$ is odd. But
$K[1]$ is also odd, because $M(Z')$ is even, and therefore 
$\widetilde M(X)$ is odd by \ref{dmevenodd}.

Now let $X$ be a projective singular curve. By base change \ref{basechange} we may
assume that $X$ has a rational point $P$. By \ref{splitl}, $M(X)=
M(X-P)\oplus \Lef$. But $X-P$ is an affine curve, and we have seen above
that it is Kimura-finite. Therefore $M(X)$ is Kimura-finite.
\end{proof}

\begin{rem}
V. Guletski\u{\i} has independently obtained this result (and also
\ref{dmevenodd}) in his recent preprint \cite{gul}.
\end{rem}

\medskip

Let $\DM_{gm}=\DM^{eff}_{gm}(k,\Q)$ be the category of effective geometrical
motives (see \cite[2.1.1]{tri}). Recall that there is a fully
faithful {\qtf} from $\DM_{gm}$ to $\DM$ and that
$\DM_{gm}$ contains the motives of all smooth schemes. 
Let $d_{\leq i}=d_{\leq i}\DM_{gm}^{eff}(k,\Q)$
be the thick subcategory of $\DM_{gm}$ generated by the
motives of all smooth schemes $X$ of dimension less or equal to $i$
(cf. \cite[p. 215]{tri}).

By \ref{2of3dm} and \ref{sfgen}, the category $d_{\leq i}$ is
Schur-finite if and only if every smooth motive $M(X)$ is Schur-finite ($\dim
X\leq i$). This observation, together with \ref{smoothcurvekf},
implies the following statement.

\begin{cor}\label{dleq1}
The category $d_{\leq 1}$ is Schur-finite.
\end{cor}

\begin{rem}
F. Orgogozo proved in \cite{org} that $d_{\leq 1}$ 
is equivalent to $\bfD^b(1\text{-mot}_\Q)$, the
bounded derived category of $1$-motives modulo isogenies. P. O'Sullivan proved that
all objects of $\bfD^b(1\text{-mot}_\Q)$ are Kimura-finite
(in $\DM_{gm}$) using the weight filtration on $1$-motives. This implies
that $d_{\leq 1}$ is actually Kimura-finite.
\end{rem}

Recall from \cite[4.3.7]{tri} that every object $A$ in $\DM_{gm}$ has a
dual $A^* =\underline{Hom}_{\DM}(A,\Z)$, where $\underline{Hom}_{\DM}$ is the internal Hom-object
of $\DM$.
By \cite{tri}, every variety has also a motive with compact support
$M^c(X)$ associated to it. If $X$ is proper, then $M(X)\isom M^c(X)$.
If $X$ is smooth of dimension $d$, then $M^c(X)\isom M(X)^*(d)[2d]$.
%The following lemma is elementary; see \cite[1.18]{delschur}.

\begin{lem}\label{schurdual}
The subcategory of $\DM_{gm}$ consisting of Schur-finite objects
is thick and closed under duals and twists.
\end{lem}

\begin{proof}
The subcategory is thick and closed under twists 
by \ref{sfthick}. If $M$ is an object of $\DM_{gm}$ 
then $(S_\lambda(M))^*\isom S_\lambda(M^*)$ by \cite[1.18]{delschur}.
In particular $M$ is Schur-finite if and only if $M^*$ is so.
\end{proof}

Before we state our results, let us investigate further 
the structure of the categories $d_{\leq i}$. 
We will write $D_{\leq n}$ for the thick subcategory
of $\DM_{gm}$ generated by the motives of all smooth projective
varieties of dimension at most $n$. If the ground field
admits resolution of singularities, then
we have the following facts.

\begin{lem}\label{lemma1}
Assume that the ground field $k$ admits resolution
of singularities in dimension $\leq n$. If $X$ is a projective
variety of dimension less or equal to $n$, then 
$M(X)^*(n)[2n]=M(X)^*\tens \Lef^n$ is 
in $D_{\leq n}$.
\end{lem}

\begin{proof}
We will proceed by induction on $d=\dim X$. If $M$ is in $D_{\leq d}$
then $M\tens \Lef^{n-d}$ is in $D_{\leq n}$, so we may assume $d=n$.
The case $n=0$ is clear.  Let us assume that the statement holds for all varieties of
dimension $n-1$ or less.  Let $Z$ be the singular locus of $X$. Using
resolution of singularities we have a smooth projective variety $X'$
and a triangle
\[ M(Z')\to M(X')\oplus M(Z)\to M(X)\to M(Z')[1].\]
Dualizing and tensoring with $\Lef^n$ we have the following triangle
\[ M(X)^*\tens \Lef^n \to (M(X')^*\tens \Lef^n) \oplus ( M(Z)^*\tens \Lef^n) 
\to M(Z')^*\tens \Lef^n\to M(X)^*\tens \Lef^n[1].\]
Both $Z$ and $Z'$ are of lower dimension, so $M(Z)^*\tens
\Lef^n$ and $M(Z')^*\tens \Lef^n$ are in $D_{\leq n}$ by induction. But
since $X$ is smooth and projective of dimension $n$,
$M(X)^*(n)[2n]=M(X)$, which is in $D_{\leq n}$ by definition.  By the
two out of three property, $M(X)^*\tens \Lef^n$ is in $D_{\leq n}$ as
well.
\end{proof}

\begin{prop}\label{dleqsingular}
If $k$ admits resolution of
singularities in dimension $\leq n$, then the category $\dleqn$
\begin{enumerate}
\item contains $M(X)$ and $M^c(X)$ 
for every variety $X$ with $\dim X\leq n$;
\item is equal to $D_{\leq n}$, i.e., it is generated by the motives
of smooth projective varieties of dimension $\leq n$.
\end{enumerate}
\end{prop}

\begin{proof}
We will proceed by induction. The case $n=0$ is clear; let us assume that the statement
holds for $n-1$.  

Let us prove the first statement. 
Let $X$ be an $n$-dimensional variety and let $Z$ be a divisor containing
its singular locus.  Using resolution of singularities, we know that
there exist a smooth $X'$ and a proper map $p:X'\to X$ which is an
isomorphism outside $Z$. From
\cite[Prop. 4.1.3]{tri}, we have an exact triangle
\[ M(Z')\to M(Z)\oplus M(X')\to M(X)\to M(Z')[1].\]
Since both $Z$ and $Z'$ are of lower dimension,
then $M(Z)$ and $M(Z')$ are in $d_{\leq n}$ by induction. But
$X'$ is smooth and therefore $M(X')$ is in $d_{\leq n}$.
By thickness, we conclude that $M(X)$ is in $d_{\leq n}$ as well.

The proof for the motives with compact support is now elementary.
For every $X$, consider a projective closure $\bar X$ and 
the complement $Z=\bar X-X$. We have a triangle
\[ M^c(Z)\to M^c(\bar X) \to M^c(X)\to M^c(Z)[1].\]
Since $\bar X$ is projective, $M^c(\bar X)=M(\bar X)$
and $M^c(Z)= M(Z)$. But $M(\bar X)$ is in
$\dleqn$ by the first part of this proof, and
$ M(Z)$ is in $\dleqn$ by induction. By thickness,
$M^c(X)$ is in $\dleqn$ as well.

And now we prove the second statement.
Clearly, $D_{\leq n}\subseteq \dleqn$ and we need to
prove that $M(X)$ is in $D_{\leq n}$ for every smooth
$X$, $n=\dim X$. Using resolution of
singularities, we may embed $X$ into a smooth projective variety $\bar X$.
Let $Z$ be the complement $\bar X-X$, and consider the triangle
\[ M(X)\to M(\bar X)\to M^c(Z)^*(n)[2n]\to M(X)[1].\]
Since $Z$ is projective,  
$M^c(Z)^*(n)[2n]$ is in $D_{\leq n}$ 
by \ref{lemma1}. Since $M(\bar X)$ is in
$D_{\leq n}$, we conclude by thickness that $M(X)$ is in $D_{\leq n}$.
\end{proof}

The following corollaries are obtained from \ref{dleqsingular}
and \ref{2of3dm}.

\begin{cor}\label{cor1}
If the motive of every smooth projective surface is Schur-finite, 
then the motive of every surface is Schur-finite.
\end{cor}

% \begin{proof}
% Combine \ref{dleq1} and \ref{induction2}.
% %By \ref{curvekf}, the motive of every curve is Kimura-finite, hence Schur-finite.
% %We conclude by \ref{induction}.
% \end{proof}

\begin{exam}
The proof shows that if $U$ is an open subset of a projective surface
$X$, and $M(X)$ is Schur-finite, then $M(U)$ is Schur-finite.
\end{exam}

\begin{cor}\label{induction2}
Assume that $k$ admits resolution of
singularities in dimension $\leq n$. 
If the motive of every
smooth projective variety of dimension less or equal to $n$ is Schur-finite,
then the motive of every variety of dimension less or equal to
$n$ is Schur-finite.
\end{cor}

%% \begin{cor}\label{cor2}
%% Assume that $k$ admits resolution of singularities.
%% If the motive of every smooth projective variety is Schur-finite,
%% then the motive of every variety is Schur-finite.
%% \end{cor}

%% \begin{rem}
%% C. Pedrini has conjectured that the two of of three
%% property holds in $\DM$ for Kimura-finiteness as well. Clearly, this
%% conjecture implies that \ref{cor1}-\ref{cor2} hold for Kimura-finiteness as
%% well.
%% \end{rem}

\subsection{An example of a motive which is not Kimura-finite}

This subsection is based on a private communication from O'Sullivan.
We will show that there is a smooth surface $U$ whose motive is Schur-finite
but not Kimura-finite.

\begin{thm}\label{nkfmot}
  (O'Sullivan) Let $X_0$ be a connected, smooth, and projective
  surface over an algebraically closed field $k_0$ such that $q=0$ and
  $p_g>0$.  Let $k=k_0(X_0)$ be the function field of $X_0$ and let
  $x_0$ be a $k_0$-point of $X_0$. Let $z$ be the zero-cycle which is
  the pullback of the cycle $\Delta(X_0)-(x_0\times X)$ along
  $X_0\times k\to X_0\times X_0$, and write $Z$ for the support of
  $z$.  Let $U$ be the complement of $Z$ in $X=X_0\times k$.  Then
  $M(U)$ is not Kimura-finite.
\end{thm}

Let $X$ be any connected, smooth, and projective surface. Let $Z$ be a
subset of $n$ $k$-rational points on $X$, and let $U=X-Z$. Then from
\cite[p.\ 196]{tri}, we have a distinguished triangle
\[ M(Z)(2)[3]\to M(U)\to M(X)\to M(Z)(2)[4]\]
The motive of $X$ decomposes as
$M(X)=\I\oplus h_1(X)\oplus h_2(X)\oplus h_3(X)\oplus \Lef^{\tens 2}$ and it
is known that the composite $\Lef^{\tens 2}\to M(X)\to M(Z)(2)[4]\isom
\oplus_n \Lef^{\tens 2}$ is the diagonal map, so we may split off one copy of
$\Lef^{\tens 2}$. This yields the following
triangle:
\[ \oplus_{n-1}\Lef^{\tens 2}[-1]\to M(U)\to 
\I\oplus h_1(X)\oplus h_2(X)\oplus h_3(X) \to \oplus_{n-1}\Lef^{\tens 2}.\]
If $q=0$ in $X$, the Picard and the Albanese varieties vanish, and then
$h_1(X)=h_3(X)=0$. So we have:
\begin{equation}\label{ostri}
\oplus_{n-1}\Lef^{\tens 2}[-1]\rTo M(U)\rTo 
\I\oplus  h_2(X)
\rTo^\partial \oplus_{n-1}\Lef^{\tens 2}.
\end{equation}

We need the following technical lemma to prove \ref{nkfmot}.

\begin{lem}\label{os}
(O'Sullivan) Let $\calD$ be a $\Q$-linear rigid tensor triangulated
category with $t$-structure and associated cohomological functor
$H^*_\tau$ with Tannakian heart and let
\[ A\rTo^f B\rTo C\rTo A[1]\]
be a distinguished triangle. Suppose that $A$ and $B$ are both even
(or both odd) and that $C$ is Kimura-finite. Then if $H^*_\tau f=0$
then $f=0$ in $\calD$.
\end{lem}

%% \begin{proof}
%% By shifting the triangle, we may assume that $A$ and $B$ are even, and 
%% that $C=C_0\oplus C_1$, where $C_0$ is even and $C_1$ is odd. 
%% Since $H^*_\tau f$ is zero, for all $i$ we have exact sequences
%% \[ 0\to H^{2i}(B)\to H^{2i}(C)\to H^{2i+1}(A)\to 0.\]
%% But if $X$ is even (resp., odd), then $H^{2i+1}(X)=0$ (resp.,
%% $H^{2i}(X)=0$) for all $i$, because the heart is Tannakian. Hence
%% $H^{2i+1}(A)=0$, and so
%% \[ H^{2i}(B)\isom H^{2i}(C_0)\isom H^{2i}(C),\]
%% for all $i$. Therefore $B\to C$ induces $H^*(B)\isom H^*(C_0)$. The
%% t-structure implies that $B\isom C_0$.  This gives a splitting of the
%% map $B\to C$, and therefore $f=0$.
%% \end{proof}

\begin{proof}[Proof of \ref{nkfmot}.]
Let $\calD$ be the derived category of $l$-adic sheaves over $\Spec k$ (see
\cite{ekedahl}). There is a {\qtf} $R\Gamma$ from $\DM$ to $\calD$
associated to $l$-adic cohomology (see \cite{huber} and \cite{hubercorrigendum}).
If we prove that the image of a
motive in $\calD$ is not Kimura-finite, it will prove that the motive
itself is not Kimura-finite in $\DM$.

It is known that $\calD$ has a t-structure whose heart is the
(Tannakian) category of $l$-adic sheaves. Moreover, $l$-adic
cohomology is the composition of $R\Gamma$ with the cohomological
functor $H_\tau^*$ associated to this t-structure.  Consider the image
of triangle (\ref{ostri}) in $\calD$. It is known that $\partial$
induces the zero map on $l$-adic cohomology, i.e.,
$H_\tau^*R\Gamma(\partial)=0$. Therefore we only need to show
that $f=R\Gamma(\partial)\not=0$ and apply 
lemma \ref{os}, to prove that $M(U)$ is not
Kimura-finite in $\DM$.

To prove that $R\Gamma(\partial)\not=0$ we proceed as follows. 
Since $p_g>0$, the map $H^2(X,\Q_l)\to H^2(\Spec k,\Q_l)$
is not zero, and so is the composition
\[\varphi: \Q_l(-2)\to H^2(X,\Q_l)\tens H^2(X,\Q_l)\to H^2(X,\Q_l)\tens
H^2(\Spec k,\Q_l)\to H^4(X,\Q_l).\]
But $\mathrm{Hom}_{Vect_{\Q_l}}(\Q_l(-2), H^4(X,\Q_l))\isom
\mathrm{Hom}_\calD(R\Gamma(\Lef^{\tens 2}),R\Gamma(M(X)))$ and therefore $\varphi$
defines a non-zero map $\psi:R\Gamma(\Lef^{\tens 2})\to R\Gamma(M(X))$.
Since $\psi$ factors through $j:R\Gamma(\oplus_n \Lef^{\tens 2})\to R\Gamma(M(X))$,
we have that $j\not=0$. Thus $R\Gamma(\partial)\not=0$.
%% By \cite[p. 178]{scholl},
%% $\Hom(h_2(X),\Lef^2)$ is isomorphic to the kernel of the Abel-Jacobi map
%% $A_0(X)\to Alb(X)$, which is non-zero if $p_g(X)> 0$ (see \cite{mumfordkyoto}).
%% Hence wen might choose an $f$ which is not zero. Hence we conclude that $\tilde M(U)$
%% cannot be Kimura-finite because it would contradict the lemma \ref{os}.
\end{proof}

\begin{cor}\label{sfnotkfmot}
  Let $X_0$ be a Kummer surface in \ref{nkfmot}. Then, using the
  notations of \ref{nkfmot}, $M(U)$ is Schur-finite but it is not
  Kimura-finite. Moreover, the two out of three property does not
  hold for Kimura-finiteness.
\end{cor}

\begin{proof}
The scheme $X_0$ satisfies the condition of theorem \ref{nkfmot}, therefore
$M(U)$ is not Kimura-finite. However, from \cite[Theorem 11 (ii)]{gp2}, 
$M(X_0)$ is Kimura-finite, and so is $M(X)$.
We conclude by applying \ref{2of3dm} to (\ref{ostri}).
\end{proof}

\nocite{orange}

\bibliographystyle{amsalpha}
\bibliography{paper}

\def\cprime{$'$}
\providecommand{\bysame}{\leavevmode\hbox to3em{\hrulefill}\thinspace}
\providecommand{\MR}{\relax\ifhmode\unskip\space\fi MR }
\providecommand{\MRhref}[2]{%
  \href{http://www.ams.org/mathscinet-getitem?mr=#1}{#2}
}
\providecommand{\href}[2]{#2}
\begin{thebibliography}{{V. }00}

\bibitem[AK02]{yveskahn}
Y.~Andr{\'e} and B.~Kahn, \emph{Nilpotence, radicaux et structures mono\"\i
  dales}, Rend. Sem. Mat. Univ. Padova \textbf{108} (2002), 107--291, with an
  appendix by P. O'Sullivan. \MR{1 956 434}

\bibitem[Bou89]{bourbaki}
N.~Bourbaki, \emph{Algebra. {I}. {C}hapters 1--3}, Elements of Mathematics,
  Springer-Verlag, Berlin, 1989, Translated from the French, Reprint of the
  1974 edition. \MR{90d:00002}

\bibitem[Del02]{delschur}
P.~Deligne, \emph{Cat\'egories tensorielles}, Mosc. Math. J. \textbf{2} (2002),
  no.~2, 227--248, Dedicated to Yuri I. Manin on the occasion of his 65th
  birthday. \MR{1 944 506}

\bibitem[Eke90]{ekedahl}
T.~Ekedahl, \emph{On the adic formalism}, The Grothendieck Festschrift, Vol.\
  II, Progr. Math., vol.~87, Birkh\"auser Boston, Boston, MA, 1990,
  pp.~197--218. \MR{92b:14010}

\bibitem[GP02]{gp}
V.~Guletski\u{\i} and C.~Pedrini, \emph{The {C}how motive of the {G}odeaux
  surface}, Algebraic geometry, de Gruyter, Berlin, 2002, pp.~179--195. \MR{1
  954 064}

\bibitem[GP03]{gp2}
V.~Guletski{\u\i} and C.~Pedrini, \emph{Finite-dimensional motives and the
  conjectures of {B}eilinson and {M}urre}, $K$-Theory \textbf{30} (2003),
  no.~3, 243--263, Special issue in honor of Hyman Bass on his seventieth
  birthday. Part III. \MR{MR2064241}

\bibitem[Gula]{Gul2}
V.~Guletski\u{\i}, \emph{{A remark on nilpotent correspondences}}, {Preprint,
  January 27, 2004, K-theory Preprint Archives,
  http://www.math.uiuc.edu/K-theory/0651/}.

\bibitem[Gulb]{gul}
\bysame, \emph{Finite dimensional objects in distinguished triangles},
  Preprint, January 5, 2004, K-theory Preprint Archives,
  \texttt{http://www.math.uiuc.edu/K-theory/0637/}.

\bibitem[Hub00]{huber}
A.~Huber, \emph{Realization of {V}oevodsky's motives}, J. Algebraic Geom.
  \textbf{9} (2000), no.~4, 755--799. \MR{2002d:14029}

\bibitem[Hub04]{hubercorrigendum}
\bysame, \emph{Corrigendum to: ``{R}ealization of {V}oevodsky's motives'' [{J}.
  {A}lgebraic {G}eom. {\bf 9} (2000), no. 4, 755--799]}, J. Algebraic Geom.
  \textbf{13} (2004), no.~1, 195--207. \MR{2 008 720}

\bibitem[Jan00]{jannsenbanff}
U.~Jannsen, \emph{Equivalence relations on algebraic cycles}, The arithmetic
  and geometry of algebraic cycles (Banff, AB, 1998), NATO Sci. Ser. C Math.
  Phys. Sci., vol. 548, Kluwer Acad. Publ., Dordrecht, 2000, pp.~225--260.
  \MR{2001f:14016}

\bibitem[Kim]{kimura}
S.-I. Kimura, \emph{Chow motives can be finite-dimensional, in some sense.}, To
  appear in J. of Alg. Geom.

\bibitem[Man68]{manin}
Ju. Manin, \emph{Correspondences, motifs and monoidal transformations}, Mat.
  Sb. (N.S.) \textbf{77 (119)} (1968), 475--507. \MR{41 \#3482}

\bibitem[MVW]{notes}
C.~Mazza, V.~Voevodsky, and C.~Weibel, \emph{Lecture notes on motivic
  cohomology}, Preprint available at
  \texttt{http://www.math.rutgers.edu/\~{}weibel/motiviclectures.html}.

\bibitem[Org04]{org}
F.~Orgogozo, \emph{Isomotifs de dimension inf\'erieure ou \'egale \'a un},
  Manuscripta Mathematica \textbf{Online First} (2004).

\bibitem[{V. }00]{orange}
{V. Voevodsky, A. Suslin and E. M. Friedlander}, \emph{Cycles, transfers, and
  motivic homology theories}, Annals of Mathematics Studies, vol. 143,
  Princeton University Press, 2000.

\bibitem[Voe00]{tri}
V.~Voevodsky, \emph{{Triangulated Categories of Motives Over a Field}}, in VSF
  \cite{orange}, pp.~188--254.

\end{thebibliography}

% \begin{thebibliography}{99}

% \bibitem[Bou]{bourbaki} Bourbaki, ``Comm. Algebra'', S-V 1989

% \bibitem[Man]{manin} manin, Motives and monoidal transformations or something

% \bibitem[Del90]{delgroth} 
% Deligne, P. Categories tannakiennes. (French) [Tannakian categories]  The Grothendieck Festschrift, Vol. II,  111--195, Progr. Math., 87, Birkhauser Boston, Boston, MA, 1990.

% \bibitem[Del02]{delschur} 
% Deligne, P. Categories tensorielles. (French) [Tensor categories] Dedicated to Yuri I. Manin on the occasion of his 65th birthday.  Mosc. Math. J.  2  (2002),  no. 2, 227--248. 

% \bibitem[Jan92]{jannsen} Jannsen,
% Inventiones Matematicae 

% \bibitem[Mur93]{murre} Murre, Indagationes
% Mate.  

% \bibitem[VSF00]{vsf} Voe, Suslin, Fried

% \bibitem[DM82]{dm} Deligne and Mumford, Hodge Cycles, Motives and....  

% \bibitem[AK02]{yveskahn} Y. Andre and B. Kahn

% \bibitem[Vo95]{voe} Voevodsky, Nilpotence theorem...

% \bibitem[FH91]{fh} Fulton - Harris, Representation Theory, Springer

% \bibitem[Jan??]{jannsenbanff} Jannsen, ``Equivalence relations on algebraic cycles''

% \bibitem[Scholl]{scholl} Scholl, ``Classical Motives''

% \bibitem[Kimura]{kimura} Kimura, S.-I. Chow motives can be
% finite-dimensional, in some sense. To appear in J. of Alg. Geom.

% \bibitem[GulPed]{gp}
% Guletski\u{\i}, V. e Pedrini, C.
% The Chow motive of the Godeaux surface
% Algebraic Geometry. A volume in Memory of Paolo Francia. 
% Walter de Gruyter 2002

% \bibitem[TriCat]{trica} V.V. Triangulated categories .... in VSF00

% \bibitem[MVW]{notes} Mazza, Voevodsky, Weibel ``Lecture notes''

% \end{thebibliography}

\end{document}